\documentclass[12pt,fleqn,twoside]{article}%
\usepackage{amsfonts}
\usepackage{espcrc1}%
\usepackage{amsmath}%
\usepackage{amssymb}%
\usepackage{graphicx}
\providecommand{\U}[1]{\protect\rule{.1in}{.1in}}
\newtheorem{theorem}{Theorem}

\newtheorem{corollary}[theorem]{Corollary}

\newtheorem{lemma}[theorem]{Lemma}

\newtheorem{remark}[theorem]{Remark}

\newenvironment{proof}[1][Proof]{\noindent\textbf{#1.} }{\ \rule{0.5em}{0.5em}}
\begin{document}
%
\title{On Approximate Controllability of Impulsive Linear Evolution Equations}%
%

\author{N. I. Mahmudov \\{Department of Mathematics \\
Eastern Mediterranean University,
Gazimagusa\\
T.R. North Cyprus, Mersin 10, TURKEY
\\
email: nazim.mahmudov@emu.edu.tr}       }%
%

\maketitle
%

\begin{abstract}
In this paper, we study an approximate controllability for the impulsive
linear evolution equations in Hilbert spaces. The necessary and sufficient
conditions for approximate controllability in terms of resolvent operators are
given. An example is provided to illustrate the application of the obtained
results.%
\end{abstract}%


\section{Introduction}

In recent years, there has been growing interests towards the study of
controllability of impulsive systems, that is, the systems in which the
system-state is subject to impulse at discrete time points. This topic has
popularity and quite broad literature, see for example, \cite{laks},
\cite{pandit}, \cite{guan1}-\cite{zhao}. The present paper studies the
necessary and sufficient conditions for approximate controllability of
dynamical systems described by linear impulsive differential equations in
Hilbert spaces, under the basic assumption that the operator $A$ acting on the
state is the infinitesimal generator of a strongly continuous semigroup.
Approximate controllability means we can control a transfer from an arbitrary
point to a small neighborhood of any other point. It can be explained by the
fact that in infinite-dimensional spaces, there are linear non closed
subspaces, see \cite{triggiani}. By looking at the approximate controllability
problem as the limit of optimal control problems and reformulating the optimal
control problem in terms of the convergence of resolvent operators, we found
the necessary and sufficient conditions for the for approximate
controllability of the impulsive linear evolution systems. The condition in
terms of resolvent is easy to use and it is used in the number of articles
devoted to the approximate controllability for different semilinear
differential equations. In the absence of impulses, so-called resolvent
condition is equivalent to the approximate controllability of the associated
linear part of the semilinear evolution control system (see \cite{basmah}%
,\ \cite{mah}), but the problem becomes complicated in the presence of impulses.

In this article, we investigate approximate controllability of the following
linear evolution systems with impulse effects%
\begin{equation}
\left\{
\begin{tabular}
[c]{ll}%
$x^{\prime}\left(  t\right)  =Ax\left(  t\right)  +Bu\left(  t\right)  ,\ $ &
$t\in\left[  0,b\right]  \backslash\left\{  t_{1},...,t_{p}\right\}  ,$\\
$\Delta x\left(  t_{k}\right)  =C_{k}x\left(  t_{k}\right)  +D_{k}v_{k},$ &
$\ t=t_{k},\ \ k=1,...,p,$\\
$x\left(  0\right)  =x_{0},$ &
\end{tabular}
\ \ \ \ \ \ \ \ \ \right.  \label{im1}%
\end{equation}
where the state variable $x\left(  \cdot\right)  $ takes values in Hilbert
space $H$ with the norm $\left\Vert \cdot\right\Vert =\sqrt{\left\langle
\cdot,\cdot\right\rangle }$. The control function $u\left(  \cdot\right)  $ is
given in $L^{2}\left(  \left[  0,b\right]  ,U\right)  $, a Hilbert space of
admissible control functions with $U$ as a Hilbert space, $v_{k}\in U,$
$k=1,...,p$. $A$ is the infinitesimal generator of a strongly continuous
semigroup of bounded linear operators $S\left(  t\right)  $ in $H$, $B\in
L\left(  U,H\right)  $, $C_{k}\in L\left(  H,H\right)  $, $D_{k}\in L\left(
U,H\right)  .$ $\Delta x\left(  t_{k}\right)  =x\left(  t_{k}^{+}\right)
-x\left(  t_{k}^{-}\right)  $ where $x\left(  t_{k}^{\pm}\right)
=\lim_{h\rightarrow0^{\pm}}x\left(  t_{k}+h\right)  $ with discontinuity
points $t_{k},$ $k=1,...,p,$ $0=t_{0}<t_{1}<t_{2}<...<t_{p}<t_{p+1}=b$. It is
assumed that $x\left(  t_{k}^{-}\right)  =x\left(  t_{k}\right)  .$

\begin{lemma}
\label{Lem:1}The mild solution of (\ref{im1}) is given by%
\begin{align*}
x\left(  t\right)   &  =S\left(  t\right)  x\left(  0\right)  +%
{\displaystyle\int_{0}^{t}}
S\left(  t-s\right)  Bu\left(  s\right)  ds,\ \ 0\leq t\leq t_{1},\\
x\left(  t\right)   &  =S\left(  t-t_{k}\right)  x\left(  t_{k}^{+}\right)  +%
{\displaystyle\int_{t_{k}}^{t}}
S\left(  t-s\right)  Bu\left(  s\right)  ds,\ \ t_{k}<t\leq t_{k+1}%
,\ \ k=1,2,...,p,
\end{align*}
where
\begin{align*}
x\left(  t_{k}^{+}\right)   &  =%
{\displaystyle\prod_{j=k}^{1}}
S_{C}\left(  t_{j},t_{j-1}\right)  x_{0}+%
{\displaystyle\sum_{i=1}^{k}}
{\displaystyle\prod_{j=k}^{i+1}}
S_{C}\left(  t_{j},t_{j-1}\right)
{\displaystyle\int_{t_{i-1}}^{t_{i}}}
S_{C}\left(  t_{i},s\right)  Bu\left(  s\right)  ds\\
&  +%
{\displaystyle\sum_{i=2}^{k}}
{\displaystyle\prod_{j=k}^{i}}
S_{C}\left(  t_{j},t_{j-1}\right)  D_{i-1}u_{i-1}+D_{k}u_{k},\ \ S_{C}\left(
t_{j},t_{j-1}\right)  :=\left(  I+C_{j}\right)  S\left(  t_{j}-t_{j-1}\right)
.
\end{align*}

\end{lemma}

\begin{proof}
The finite dimensional case is proved in \cite{guan2}. The proof is similar to
that of Lemma 2.1 in \cite{guan2}.
\end{proof}

Associated with (\ref{im1}) ( $B=0,$ $D_{k}=0$), consider it's adjoint
equation given by%
\begin{equation}
\left\{
\begin{tabular}
[c]{l}%
$\psi^{\prime}\left(  t\right)  =-A^{\ast}\psi\left(  t\right)  ,$\\
$\psi\left(  b\right)  =\varphi,$\\
$\Delta\psi\left(  t_{p-k+1}\right)  =-C_{p-k+1}^{\ast}\psi\left(
t_{p-k+1}^{+}\right)  ,\ \ k=1,...,p.$%
\end{tabular}
\ \ \ \ \ \ \ \ \right.  \label{ad1}%
\end{equation}
Here $A^{\ast},$ $C_{p-k+1}^{\ast}$ are adjoint operators.

\begin{lemma}
\label{Lem:2}The mild solution of the adjoint equation (\ref{ad1}) is given by%
\begin{align}
\psi\left(  t\right)   &  =S^{\ast}\left(  b-t\right)  \varphi,\ \ t_{p}<t\leq
b,\nonumber\\
\psi\left(  t\right)   &  =S_{C}^{\ast}\left(  t_{k},t\right)  \prod
_{i=k+1}^{p}S_{C}^{\ast}\left(  t_{i},t_{i-1}\right)  S^{\ast}\left(
b-t_{p}\right)  \varphi,\ \ t_{k-1}<t\leq t_{k},\ k=p,...,1,\ \ \prod
_{i=p+1}^{p}=1. \label{ad2}%
\end{align}

\end{lemma}

\begin{proof}
For $t_{p}<t\leq b$ the formula (\ref{ad2}) is obvious. For $t_{p-1}<t\leq
t_{p},$ we have%
\begin{align*}
\psi\left(  t\right)   &  =S^{\ast}\left(  t_{p}-t\right)  \psi\left(
t_{p}^{-}\right)  =S^{\ast}\left(  t_{p}-t\right)  \left(  I+C_{p}^{\ast
}\right)  \psi\left(  t_{p}^{+}\right) \\
&  =S^{\ast}\left(  t_{p}-t\right)  \left(  I+C_{p}^{\ast}\right)  S^{\ast
}\left(  b-t_{p}\right)  \varphi=S_{C}^{\ast}\left(  t_{p},t\right)  S^{\ast
}\left(  b-t_{p}\right)  \varphi.
\end{align*}
By using the induction we get the desired formula (\ref{ad2}).
\end{proof}

We now introduce the following lemma which characterizes the relationship
between solutions of (\ref{im1}), (\ref{ad1}) and the control operators.

\begin{lemma}
\label{Lem:3}For the solutions (\ref{im1}) and (\ref{ad2}), the following
formula holds:%
\begin{equation}
\left\langle x\left(  b\right)  ,\psi\left(  b\right)  \right\rangle
-\left\langle x\left(  0\right)  ,\psi\left(  0\right)  \right\rangle =%
{\displaystyle\sum_{k=1}^{p+1}}
{\displaystyle\int_{t_{k-1}}^{t_{k}}}
\left\langle u\left(  s\right)  ,B^{\ast}\psi\left(  s\right)  \right\rangle
ds+%
{\displaystyle\sum_{k=1}^{p}}
\left\langle v_{k},D_{k}^{\ast}\psi\left(  t_{k}^{+}\right)  \right\rangle .
\label{f2}%
\end{equation}

\end{lemma}

\begin{proof}
It is clear that%
\begin{equation}
\left\langle x\left(  t_{1}\right)  ,\psi\left(  t_{1}\right)  \right\rangle
-\left\langle x\left(  0\right)  ,\psi\left(  0\right)  \right\rangle =%
{\displaystyle\int_{0}^{t_{1}}}
\left\langle Bu\left(  s\right)  ,\psi\left(  s\right)  \right\rangle ds.
\label{qq1}%
\end{equation}
From Lemmas \ref{Lem:1} and \ref{Lem:2}, we have%
\begin{gather}
\left\langle x\left(  b\right)  ,\psi\left(  b\right)  \right\rangle
-\left\langle x\left(  t_{1}\right)  ,\psi\left(  t_{1}\right)  \right\rangle
=%
{\displaystyle\sum_{k=2}^{p+1}}
\left[  \left\langle x\left(  t_{k}\right)  ,\psi\left(  t_{k}\right)
\right\rangle -\left\langle x\left(  t_{k-1}\right)  ,\psi\left(
t_{k-1}\right)  \right\rangle \right] \nonumber\\
=%
{\displaystyle\sum_{k=2}^{p+1}}
\left[  \left\langle S\left(  t_{k}-t_{k-1}\right)  x\left(  t_{k-1}%
^{+}\right)  +%
{\displaystyle\int_{t_{k-1}}^{t_{k}}}
S\left(  t_{k}-s\right)  Bu\left(  s\right)  ds,\psi\left(  t_{k}\right)
\right\rangle -\left\langle x\left(  t_{k-1}\right)  ,\left(  I+C_{k-1}^{\ast
}\right)  \psi\left(  t_{k-1}^{+}\right)  \right\rangle \right] \nonumber\\
=%
{\displaystyle\sum_{k=2}^{p+1}}
{\displaystyle\int_{t_{k-1}}^{t_{k}}}
\left\langle Bu\left(  s\right)  ,S^{\ast}\left(  t_{k}-s\right)  \psi\left(
t_{k}\right)  \right\rangle ds\nonumber\\
+%
{\displaystyle\sum_{k=2}^{p+1}}
\left[  \left\langle x\left(  t_{k-1}\right)  +C_{k-1}x\left(  t_{k-1}\right)
+D_{k-1}v_{k-1},\psi\left(  t_{k-1}^{+}\right)  \right\rangle -\left\langle
x\left(  t_{k-1}\right)  ,\left(  I+C_{k-1}^{\ast}\right)  \psi\left(
t_{k-1}^{+}\right)  \right\rangle \right] \nonumber\\
=%
{\displaystyle\sum_{k=2}^{p+1}}
{\displaystyle\int_{t_{k-1}}^{t_{k}}}
\left\langle u\left(  s\right)  ,B^{\ast}\psi\left(  s\right)  \right\rangle
ds+%
{\displaystyle\sum_{k=1}^{p}}
\left\langle v_{k},D_{k}^{\ast}\psi\left(  t_{k}^{+}\right)  \right\rangle .
\label{qq2}%
\end{gather}
Combining (\ref{qq1}) and (\ref{qq2}), the formula (\ref{f2}) is obtained.
\end{proof}

\section{Main results}

For convenience, denote by $w=\left(  u\left(  \cdot\right)  ,\left\{
v_{k}\right\}  _{k=1}^{p}\right)  \in L^{2}\left(  \left[  0,b\right]
,U\right)  \times U^{p}.$ The space $L^{2}\left(  \left[  0,b\right]
,U\right)  \times U^{p}$ of $w$ is a Hilbert space with respect to the inner
product $\left\langle \cdot,\cdot\right\rangle _{1}$ is defined as
$\left\langle w_{1},w_{2}\right\rangle _{1}=%
{\displaystyle\int_{0}^{b}}
\left\langle u_{1}\left(  s\right)  ,u_{2}\left(  s\right)  \right\rangle
_{U}ds+%
{\displaystyle\sum_{k=1}^{p}}
\left\langle v_{1k},v_{2k}\right\rangle _{U}$ for all $w_{1},w_{2}\in
L^{2}\left(  \left[  0,b\right]  ,U\right)  \times U^{p}$.

To define the analogue of controllability operator for impulsive system, we
introduce the bounded linear operator $M:L^{2}\left(  \left[  0,b\right]
,U\right)  \times U^{p}\rightarrow H$ as follows%
\begin{align*}
Mw  &  =S\left(  b-t_{p}\right)
{\displaystyle\sum_{i=1}^{p}}
{\displaystyle\prod_{j=p}^{i+1}}
S_{C}\left(  t_{j},t_{j-1}\right)
{\displaystyle\int_{t_{i-1}}^{t_{i}}}
S_{C}\left(  t_{i},s\right)  Bu\left(  s\right)  ds+%
{\displaystyle\int_{t_{p}}^{b}}
S\left(  b-s\right)  Bu\left(  s\right)  ds\\
&  +S\left(  b-t_{p}\right)
{\displaystyle\sum_{i=2}^{p}}
{\displaystyle\prod_{j=p}^{i}}
S_{C}\left(  t_{j},t_{j-1}\right)  D_{i-1}v_{i-1}+S\left(  b-t_{p}\right)
D_{p}v_{p}.
\end{align*}
Letting $x\left(  0\right)  =0$ in (\ref{f2}) yields $\left\langle x\left(
b\right)  ,\varphi\right\rangle =\left\langle w,M^{\ast}\psi_{b}\right\rangle
_{1}=%
{\displaystyle\int_{0}^{b}}
\left\langle u\left(  s\right)  ,B^{\ast}\psi\left(  s\right)  \right\rangle
ds+%
{\displaystyle\sum_{k=1}^{p}}
\left\langle v_{k},D_{k}^{\ast}\psi\left(  t_{k}^{+}\right)  \right\rangle
,$which implies%
\begin{align*}
M^{\ast}\varphi &  =\left(  B^{\ast}\psi\left(  \cdot\right)  ,\left\{
D_{k}^{\ast}\psi\left(  t_{k}^{+}\right)  \right\}  _{k=1}^{p}\right)  ,\\
B^{\ast}\psi\left(  t\right)   &  =\left\{
\begin{tabular}
[c]{ll}%
$B^{\ast}S^{\ast}\left(  b-t\right)  \varphi,$ & $t_{p}<t\leq b,$\\
$B^{\ast}S_{C}^{\ast}\left(  t_{k},t\right)  \prod_{i=k+1}^{p}S_{C}^{\ast
}\left(  t_{i},t_{i-1}\right)  S^{\ast}\left(  b-t_{p}\right)  \varphi,$ &
$t_{k-1}<t\leq t_{k},$%
\end{tabular}
\ \ \ \ \right. \\
D_{k}^{\ast}\psi\left(  t_{k}^{+}\right)   &  =\left\{
\begin{tabular}
[c]{ll}%
$D_{p}^{\ast}S^{\ast}\left(  b-t_{p}\right)  \varphi,$ & $k=p,$\\
$D_{k}^{\ast}\prod_{i=k}^{p}S_{C}^{\ast}\left(  t_{i},t_{i-1}\right)  S^{\ast
}\left(  b-t_{p}\right)  \varphi,$ & $k=p-1,...,1.$%
\end{tabular}
\ \ \ \ \right.
\end{align*}
We are ready to introduce four controllability operators $\Gamma_{t_{p}}%
^{b},\widetilde{\Gamma}_{t_{p}}^{b},\Theta_{0}^{t_{p}},\widetilde{\Theta}%
_{0}^{t_{p}}:H\rightarrow H$,%
\begin{align*}
\Gamma_{t_{p}}^{b}  &  :=%
{\displaystyle\int_{t_{p}}^{b}}
S\left(  b-s\right)  BB^{\ast}S^{\ast}\left(  b-s\right)  ds,\ \ \ \widetilde
{\Gamma}_{t_{p}}^{b}:=S\left(  b-t_{p}\right)  D_{p}D_{p}^{\ast}S^{\ast
}\left(  b-t_{p}\right)  ,\\
\Theta_{0}^{t_{p}}  &  :=S\left(  b-t_{p}\right)
{\displaystyle\sum_{i=1}^{p}}
{\displaystyle\prod_{j=p}^{i+1}}
S_{C}\left(  t_{j},t_{j-1}\right)
{\displaystyle\int_{t_{i-1}}^{t_{i}}}
S_{C}\left(  t_{i},s\right)  BB^{\ast}S_{C}^{\ast}\left(  t_{i},s\right)
ds\prod_{k=i+1}^{p}S_{C}^{\ast}\left(  t_{k},t_{k-1}\right)  S^{\ast}\left(
b-t_{p}\right)  ,\\
\widetilde{\Theta}_{0}^{t_{p}}  &  :=S\left(  b-t_{p}\right)
{\displaystyle\sum_{i=2}^{p}}
{\displaystyle\prod_{j=p}^{i}}
S_{C}\left(  t_{j},t_{j-1}\right)  D_{i-1}D_{i-1}^{\ast}\prod_{k=i}^{p}%
S_{C}^{\ast}\left(  t_{k},t_{k-1}\right)  S^{\ast}\left(  b-t_{p}\right)  .
\end{align*}
It is clear that $MM^{\ast}=\Theta_{0}^{t_{p}}+\Gamma_{t_{p}}^{b}%
+\widetilde{\Theta}_{0}^{t_{p}}+\widetilde{\Gamma}_{t_{p}}^{b}.$

\begin{remark}
Note that in nonimpulsive case $MM^{\ast}=\Gamma_{0}^{b}:=%
{\displaystyle\int_{0}^{b}}
S\left(  b-s\right)  BB^{\ast}S^{\ast}\left(  b-s\right)  ds,$ we have only
one controllability operator.
\end{remark}

\begin{theorem}
\label{Thm:1}The following conditions are equivalent.

\begin{enumerate}
\item[(\ref{Thm:1}a)] System (\ref{im1}) is approximately controllable on
$\left[  0,b\right]  $.

\item[(\ref{Thm:1}b)] $M^{\ast}\varphi=0$ implies that $\varphi=0.$

\item[(\ref{Thm:1}c)] $\Theta_{0}^{t_{p}}+\Gamma_{t_{p}}^{b}+\widetilde
{\Theta}_{0}^{t_{p}}+\widetilde{\Gamma}_{t_{p}}^{b}$ is positive.

\item[(\ref{Thm:1}d)] $\varepsilon\left(  \varepsilon I+\Theta_{0}^{t_{p}%
}+\Gamma_{t_{p}}^{b}+\widetilde{\Theta}_{0}^{t_{p}}+\widetilde{\Gamma}_{t_{p}%
}^{b}\right)  ^{-1}$ converges to zero operator as $\varepsilon\rightarrow
0^{+}$ in strong operator topology.

\item[(\ref{Thm:1}e)] $\varepsilon\left(  \varepsilon I+\Theta_{0}^{t_{p}%
}+\Gamma_{t_{p}}^{b}+\widetilde{\Theta}_{0}^{t_{p}}+\widetilde{\Gamma}_{t_{p}%
}^{b}\right)  ^{-1}$ converges to zero operator as $\varepsilon\rightarrow
0^{+}$ in weak operator topology.
\end{enumerate}
\end{theorem}

\begin{proof}
The proof of the equivalence (\ref{Thm:1}a)$\Longleftrightarrow$(\ref{Thm:1}b)
is standard. Approximately controllability of system (\ref{im1}) on $\left[
0,b\right]  $ is equivalent to Im$M$ is dense in $H$. That means, the kernel
of $M^{\ast}$ is trivial in $H$. Equivalently, $M^{\ast}\varphi=\left(
B^{\ast}\psi\left(  \cdot\right)  ,\left\{  D_{k}^{\ast}\psi\left(  t_{k}%
^{+}\right)  \right\}  _{k=1}^{p}\right)  =0$ implies that $\varphi=0.$ The
equivalence (\ref{Thm:1}a)$\Longleftrightarrow$(\ref{Thm:1}c) is well known,
see \cite{zabczyk} page 207. The equivalence (\ref{Thm:1}%
d)$\Longleftrightarrow$(\ref{Thm:1}e) is a consequence of nonnegativity of
$\varepsilon\left(  \varepsilon I+\Theta_{0}^{t_{p}}+\Gamma_{t_{p}}%
^{b}+\widetilde{\Theta}_{0}^{t_{p}}+\widetilde{\Gamma}_{t_{p}}^{b}\right)
^{-1}$. We prove only (\ref{Thm:1}a)$\Longleftrightarrow$(\ref{Thm:1}d). To do
so, consider the functional%
\[
J_{\varepsilon}\left(  \varphi\right)  =\frac{1}{2}\left\Vert M^{\ast}%
\varphi\right\Vert ^{2}+\frac{\varepsilon}{2}\left\Vert \varphi\right\Vert
^{2}-\left\langle \varphi,h-S\left(  b-t_{p}\right)
{\displaystyle\prod_{j=p}^{1}}
S_{C}\left(  t_{j},t_{j-1}\right)  x_{0}\right\rangle
\]
The map $\varphi\rightarrow J_{\varepsilon}\left(  \varphi\right)  $ is
continuous and strictly convex. The functional $J_{\varepsilon}\left(
\cdot\right)  $ admits a unique minimum $\widehat{\varphi}_{\varepsilon}$ that
defines a map $\Phi:X\rightarrow X$. Since $J_{\varepsilon}\left(
\varphi\right)  $ is Frechet differentiable at $\widehat{\varphi}%
_{\varepsilon},$ by the optimality of $\widehat{\varphi}_{\varepsilon}$, we
must have%
\begin{equation}
\frac{d}{d\varphi}J_{\varepsilon}\left(  \varphi\right)  =\Theta_{0}^{t_{p}%
}\widehat{\varphi}_{\varepsilon}+\Gamma_{t_{p}}^{b}\widehat{\varphi
}_{\varepsilon}+\widetilde{\Theta}_{0}^{t_{p}}\widehat{\varphi}_{\varepsilon
}+\widetilde{\Gamma}_{t_{p}}^{b}\widehat{\varphi}_{\varepsilon}+\varepsilon
\widehat{\varphi}_{\varepsilon}-h+S\left(  b-t_{p}\right)
{\displaystyle\prod_{j=p}^{1}}
S_{C}\left(  t_{j},t_{j-1}\right)  x_{0}=0, \label{wq1}%
\end{equation}
By solving (\ref{wq1}) for $\widehat{\varphi}_{\varepsilon},$ we get%
\begin{equation}
\widehat{\varphi}_{\varepsilon}=\left(  \varepsilon I+\Theta_{0}^{t_{p}%
}+\Gamma_{t_{p}}^{b}+\widetilde{\Theta}_{0}^{t_{p}}+\widetilde{\Gamma}_{t_{p}%
}^{b}\right)  ^{-1}\left(  h-S\left(  b-t_{p}\right)
{\displaystyle\prod_{j=p}^{1}}
S_{C}\left(  t_{j},t_{j-1}\right)  x_{0}\right)  . \label{wq3}%
\end{equation}
Defining $u^{\varepsilon}\left(  s\right)  $ and $\left\{  v_{k}^{\varepsilon
}\right\}  _{k=1}^{p}$ as follows
\begin{align*}
u^{\varepsilon}\left(  s\right)   &  =\left(
{\displaystyle\sum_{k=1}^{p}}
B^{\ast}S_{C}^{\ast}\left(  t_{k},s\right)  \prod_{i=k+1}^{p}S_{C}^{\ast
}\left(  t_{i},t_{i-1}\right)  S^{\ast}\left(  b-t_{p}\right)  \chi_{\left(
t_{k-1},t_{k}\right)  }+B^{\ast}S^{\ast}\left(  b-s\right)  \chi_{\left(
t_{p},b\right)  }\right)  \widehat{\varphi}_{\varepsilon},\\
v_{p}^{\varepsilon}  &  =D_{p}^{\ast}S^{\ast}\left(  b-t_{p}\right)
\widehat{\varphi}_{\varepsilon},\ \ v_{k}^{\varepsilon}=D_{k}^{\ast}%
\prod_{i=k}^{p}S_{C}^{\ast}\left(  t_{i},t_{i-1}\right)  S^{\ast}\left(
b-t_{p}\right)  \widehat{\varphi}_{\varepsilon},\ k=1,...,p-1,
\end{align*}
we get from (\ref{wq1}) and (\ref{wq3}) that%
\begin{equation}
x_{\varepsilon}\left(  b\right)  -h=-\varepsilon\widehat{\varphi}%
_{\varepsilon}=-\varepsilon\left(  \varepsilon I+\Theta_{0}^{t_{p}}%
+\Gamma_{t_{p}}^{b}+\widetilde{\Theta}_{0}^{t_{p}}+\widetilde{\Gamma}_{t_{p}%
}^{b}\right)  ^{-1}\left(  h-S\left(  b-t_{p}\right)
{\displaystyle\prod_{j=p}^{1}}
S_{C}\left(  t_{j},t_{j-1}\right)  x_{0}\right)  , \label{wq2}%
\end{equation}
where
\[
x_{\varepsilon}\left(  b\right)  =x\left(  b;x_{0},u^{\varepsilon},\left\{
v_{k}^{\varepsilon}\right\}  _{k=1}^{p}\right)  =S\left(  b-t_{p}\right)
{\displaystyle\prod_{j=p}^{1}}
S_{C}\left(  t_{j},t_{j-1}\right)  x_{0}+\Theta_{0}^{t_{p}}\widehat{\varphi
}_{\varepsilon}+\Gamma_{t_{p}}^{b}\widehat{\varphi}_{\varepsilon}%
+\widetilde{\Theta}_{0}^{t_{p}}\widehat{\varphi}_{\varepsilon}+\widetilde
{\Gamma}_{t_{p}}^{b}\widehat{\varphi}_{\varepsilon}.
\]
Now, the equivalence (\ref{Thm:1}a)$\Longleftrightarrow$(\ref{Thm:1}d) follows
immediately from (\ref{wq2}).
\end{proof}

\begin{corollary}
\label{Cor:1}If one of the operators $\Theta_{0}^{t_{p}},\Gamma_{t_{p}}%
^{b},\widetilde{\Theta}_{0}^{t_{p}},\widetilde{\Gamma}_{t_{p}}^{b}$ is
positive, then $\Theta_{0}^{t_{p}}+\Gamma_{t_{p}}^{b}+\widetilde{\Theta}%
_{0}^{t_{p}}+\widetilde{\Gamma}_{t_{p}}^{b}$ is positive, and consequently the
system (\ref{im1}) is approximately controllable on $\left[  0,b\right]  .$
\end{corollary}

\begin{proof}
The operators $\Theta_{0}^{t_{p}},\Gamma_{t_{p}}^{b},\widetilde{\Theta}%
_{0}^{t_{p}},\widetilde{\Gamma}_{t_{p}}^{b}$ are nonnegative. So, if one of
them is positive then the sum $\Theta_{0}^{t_{p}}+\Gamma_{t_{p}}%
^{b}+\widetilde{\Theta}_{0}^{t_{p}}+\widetilde{\Gamma}_{t_{p}}^{b}$ is
positive. By Theorem \ref{Thm:1}c, we get the approximate controllabiliy of
(\ref{im1}) on $\left[  0,b\right]  .$
\end{proof}

\begin{corollary}
Assume $A:H\rightarrow H$ is a linear bounded operator. System (\ref{im1}) is
approximately controllable on $\left[  0,b\right]  $ if
\begin{equation}
\overline{sp\left\{  A^{n}BU:n=0,1,2,...\right\}  }=H. \label{apc1}%
\end{equation}

\end{corollary}

\begin{proof}
Suppose by contradiction that%
\[
\text{Im}M=\left\{  x\left(  b\right)  =x\left(  b;0,u,\left\{  v_{k}\right\}
_{k=1}^{p}\right)  :\left(  u,\left\{  v_{k}\right\}  _{k=1}^{p}\right)  \in
L^{2}\left(  \left[  0,b\right]  ,U\right)  \times U^{p}\right\}
\]
is not dense in $H,$ then for some nonzero $\varphi\in H$ $\left\langle
x\left(  b\right)  ,\varphi\right\rangle =0:$
\[
\left\langle x\left(  b\right)  ,\varphi\right\rangle =%
{\displaystyle\int_{0}^{b}}
\left\langle Bu\left(  s\right)  ,\psi\left(  s\right)  \right\rangle ds+%
{\displaystyle\sum_{k=1}^{p}}
\left\langle D_{k}v_{k},\psi\left(  t_{k}^{+}\right)  \right\rangle
=0,\ \text{for any }\left(  u,\left\{  v_{k}\right\}  _{k=1}^{p}\right)  \in
L^{2}\left(  \left[  0,b\right]  ,U\right)  \times U^{p},
\]
where $\psi$ is a solution (\ref{ad2}) of the adjoint equation with
$\varphi\neq0$. This easily leads to
\[
\left\Vert M^{\ast}\varphi\right\Vert ^{2}=\left\langle \left(  \Theta
_{0}^{t_{p}}+\Gamma_{t_{p}}^{b}+\widetilde{\Theta}_{0}^{t_{p}}+\widetilde
{\Gamma}_{t_{p}}^{b}\right)  \varphi,\varphi\right\rangle =0\Longrightarrow
\Gamma_{t_{p}}^{b}\varphi=0\Longrightarrow B^{\ast}S^{\ast}\left(  b-s\right)
\varphi=0,\ \ t_{p}\leq s\leq b.
\]
We differentiate successively this last identity to show, by induction, that
$B^{\ast}\varphi=B^{\ast}A^{\ast}\varphi=...=B^{\ast}\left(  A^{\ast}\right)
^{n}\varphi=0,\ \ n=0,1,2,...$Therefore $0\neq\varphi\in\cap_{n=0}^{\infty
}\ker\left\{  B^{\ast}\left(  A^{\ast}\right)  ^{n}\right\}  .$ But the
condition (\ref{apc1}) is equivalent to $\cap_{n=0}^{\infty}\ker\left\{
B^{\ast}\left(  A^{\ast}\right)  ^{n}\right\}  =0$ see \cite{triggiani2}. This
contradiction proves that system (\ref{im1}) is approximately controllable on
$\left[  0,b\right]  $.
\end{proof}

Nonimpulsive analogue of the following wave equation is given in
\cite{zabczyk}.

\begin{theorem}
If $b-t_{p}\geq2\pi,$ $\gamma_{m}\neq0\ $for $m=1,2,...,$ then system%
\begin{equation}
\left\{
\begin{array}
[c]{c}%
\begin{tabular}
[c]{lll}%
$\dfrac{\partial^{2}x\left(  t,\theta\right)  }{\partial t^{2}}=\dfrac
{\partial^{2}x\left(  t,\theta\right)  }{\partial\theta^{2}}+hu\left(
t\right)  ,$ &  & \\
$x\left(  t,0\right)  =x\left(  t,\pi\right)  =0,$ &  & \\
$x\left(  0,\theta\right)  =a\left(  \theta\right)  ,\ \ \dfrac{\partial
x\left(  0,\theta\right)  }{\partial t}=b\left(  \theta\right)  ,$ &  &
\end{tabular}
\\
\Delta x\left(  t_{i},\theta\right)  =a_{i}\left(  \theta\right)
,\ \ \ \Delta\dfrac{\partial x\left(  t_{i},\theta\right)  }{\partial t}%
=b_{i}\left(  \theta\right)  ,\ \ i=1,...,p
\end{array}
\right.  \label{wave}%
\end{equation}
is approximately controllable on $\left[  0,b\right]  $.
\end{theorem}

\begin{proof}
We identify functions $a\left(  \theta\right)  $ and $b\left(  \theta\right)
$ with their Fourier expansions%
\begin{equation}
a\left(  \theta\right)  =\sum_{m=1}^{\infty}\alpha_{m}\sin m\theta
,\ \ \ b\left(  \theta\right)  =\sum_{m=1}^{\infty}\beta_{m}\sin
m\theta,\ \ \theta\in\left(  0,\pi\right)  . \label{exp1}%
\end{equation}
It is easy to check that%
\[
x\left(  t,\theta\right)  =\sum_{m=1}^{\infty}\left(  \alpha_{m}\cos
mt+\frac{\beta_{m}}{m}\sin mt\right)  \sin m\theta,\ \ \dfrac{\partial
x\left(  t,\theta\right)  }{\partial t}=\sum_{m=1}^{\infty}\left(
-m\alpha_{m}\sin mt+\beta_{m}\cos mt\right)  \sin m\theta.
\]
We define $H$ to be the set of pairs $\left[
\begin{array}
[c]{c}%
a\\
b
\end{array}
\right]  $ of functions with expansions (\ref{exp1}) such that $\sum
_{m=1}^{\infty}\left(  m^{2}\left\vert \alpha_{m}\right\vert ^{2}+\left\vert
\beta_{m}\right\vert ^{2}\right)  <\infty.$ This is a Hilbert space with
scalar product $\left\langle \left[
\begin{array}
[c]{c}%
a\\
b
\end{array}
\right]  ,\left[
\begin{array}
[c]{c}%
\widetilde{a}\\
\widetilde{b}%
\end{array}
\right]  \right\rangle =\sum_{m=1}^{\infty}\left(  m^{2}\alpha_{m}%
\widetilde{\alpha}_{m}+\beta_{m}\widetilde{\beta}_{m}\right)  .$ The semigroup
of solutions to the wave equation (\ref{wave}) is defined as follows:%
\[
S\left(  t\right)  \left[
\begin{array}
[c]{c}%
a\\
b
\end{array}
\right]  =\sum_{m=1}^{\infty}\left[
\begin{tabular}
[c]{ll}%
$\cos mt$ & $\frac{1}{m}\sin mt$\\
$-m\sin mt$ & $\cos mt$%
\end{tabular}
\ \ \ \ \ \ \right]  \left[
\begin{array}
[c]{c}%
\alpha_{m}\\
\beta_{m}%
\end{array}
\right]  \sin m\left(  \cdot\right)  ,\ \ t\geq0.
\]
The formula is meaningful for all $t\in R$ and $S^{\ast}\left(  t\right)
=S^{-1}\left(  t\right)  =S\left(  -t\right)  ,\ \ t\in R.$ It is known that
the problem can be written as follows:%
\[
\left[
\begin{array}
[c]{c}%
y_{1}\left(  t\right) \\
y_{2}\left(  t\right)
\end{array}
\right]  =S\left(  t\right)  \left[
\begin{array}
[c]{c}%
a\\
b
\end{array}
\right]  +\int_{0}^{t}S\left(  t-s\right)  \left[
\begin{array}
[c]{c}%
0\\
h
\end{array}
\right]  u\left(  s\right)  ds.
\]
In our case $U=R$ and the operator $B:R\rightarrow H$ is given by $Bu=\left[
\begin{array}
[c]{c}%
0\\
h
\end{array}
\right]  u,\ \ u\in R.$ Since $S^{\ast}\left(  t\right)  =S\left(  -t\right)
,\ t\geq0,$%
\[
B^{\ast}S^{\ast}\left(  b-t\right)  \left[
\begin{array}
[c]{c}%
a\\
b
\end{array}
\right]  =\sum_{m=1}^{\infty}\gamma_{m}\left(  m\alpha_{m}\sin m\left(
b-t\right)  +\beta_{m}\cos m\left(  b-t\right)  \right)  ,\ \ t_{p}\leq t\leq
b.
\]
It is easy to see that the series on the right hand side, when is denoted by
$\phi\left(  t\right)  ,$ $0\leq t\leq b-t_{p},$ defines a continuous,
periodic function with period $2\pi$. Moreover,%
\[
m\gamma_{m}\alpha_{m}=\dfrac{1}{\pi}\int_{0}^{2\pi}\phi\left(  t\right)  \cos
mtdt,\ \gamma_{m}\beta_{m}=\dfrac{1}{\pi}\int_{0}^{2\pi}\phi\left(  t\right)
\sin mtdt,\ \ m=1,2,...\
\]
Hence if $b\geq t_{p}+2\pi$ and $\phi\left(  t\right)  =0$ for $0\leq t\leq
b-t_{p},$ then $m\gamma_{m}\alpha_{m}=0\ \ \ $and\ \ $\gamma_{m}\beta
_{m}=0,\ \ \ m=1,2,...$Since $\gamma_{m}\neq0\ $for $m=1,2,...,\ \alpha
_{m}=\beta_{m}=0,\ m=1,2,...,$ and we obtain that $a=b=0.$ By Corollary
\ref{Cor:1}, wave equation (\ref{wave}) is approximately controllable.
\end{proof}


\begin{thebibliography}{99}                                                                                               %


\bibitem {zabczyk}J. Zabczyk, Mathematical Control Theory: An Introduction,
Springer, 2008.

\bibitem {laks}V. Lakshmikantham, D.D. Bainov, P.P.S. Simeonov, Theory of
Impulsive Differential Equations, vol. 6, World Scientific, 1989.

\bibitem {pandit}S.G. Pandit, S.G. Deo, Differential Systems Involving
Impulses, Springer-Verlag, New York, 1982.

\bibitem {guan1}Z.-H. Guan, T.-H. Qian, X. Yu, On controllability and
observability for a class of impulsive systems, Syst. Control Lett. 47 (3)
(2002) 247--257.

\bibitem {guan2}Z.H. Guan, T.H. Qian, X.H. Yu, Controllability and
observability of linear time-varying impulsive systems, IEEE Trans. Circuits
Syst. -I 49 (8) (2002) 1198--1208.

\bibitem {leela}S. Leela, F. McRae, S. Sivasundaram, Controllability of
impulsive differential equations, J. Math. Anal. Appl. 177 (1) (1993) 24--30.

\bibitem {xie}G. Xie, L. Wang, Necessary and sufficient conditions for
controllability and observability of switched impulsive control systems, IEEE
Trans. Autom. Control 49 (34) (2004) 960--966.

\bibitem {xie2}G. Xie, L. Wang, Controllability and observability of a class
of linear impulsive systems, J. Math. Anal. Appl. 304 (1) (2005) 336--355.

\bibitem {zhao1}S. Zhao, J. Sun, Controllability and observability for a class
of time-varying impulsive systems, Nonlinear Anal.: Real World Appl. 10 (3)
(2009) 1370--1380.

\bibitem {liu}Y. Liu, S. Zhao, Controllability for a class of linear
time-varying impulsive systems with time delay in control input, IEEE Trans.
Autom. Control 56 (2) (2011) 395--399.

\bibitem {zhao}S. Zhao, Z. Zhang, T. Wang, W. Yu, Controllability for a class
of time-varying controlled switching impulsive systems with time delays, Appl.
Math. Comput. 228 (2014) 404--410.

\bibitem {triggiani}R. Triggiani, A note on the lack of exact controllability
for mild solutions in Banach spaces. SIAM J. Control Optim. 15(3), 407--411 (1977).

\bibitem {triggiani2}R. Triggiani, Controllability and observability in Banach
space with bounded operators. SIAM J. Control 13 (1975), 462--491.

\bibitem {basmah}A.E. Bashirov, N.I. Mahmudov, On concepts of controllability
for deterministic and stochastic systems. SIAM J. Control Optim. 37(6),
1808--1821 (1999).

\bibitem {mah}N.I. Mahmudov, Approximate controllability of semilinear
deterministic and stochastic evolution equations in abstract spaces. SIAM J.
Control Optim. 42(5), 1604--1622 (2003).
\end{thebibliography}
\end{document}